\documentclass[jsl]{aslb}

\DeclareMathOperator{\interior}{int}

\begin{document}
\newtheorem{theorem}{Theorem}[section]
\newtheorem{prop}[theorem]{Proposition}

\theoremstyle{definition}
\newtheorem{definition}[theorem]{Definition}
\newtheorem{example}[theorem]{Example}
\newtheorem{remark}[theorem]{Remark}
\newtheorem{conj}[theorem]{Conjecture}

\providecommand{\floor}[1]{\left\lfloor#1\right\rfloor}
\providecommand{\Z}{\mathbb{Z}}
\providecommand{\R}{\mathbb{R}}
\providecommand{\N}{\mathbb{N}}
\providecommand{\C}{\mathbb{C}}
\providecommand{\Q}{{\mathbb{Q}}}

\title[Presburger, generating functions, quasi-polynomials]{Presburger arithmetic, rational generating functions, and quasi-polynomials}
\author{Kevin Woods}
\revauthor{Woods, Kevin}
\address{Department of Mathematics\\Oberlin College\\Oberlin, Ohio 44074, USA}
\email{Kevin.Woods@oberlin.edu}
\urladdr{http://www.oberlin.edu/faculty/kwoods/}
\thanks{Extended abstract published in ICALP 2013.}

\begin{abstract}
Presburger arithmetic is the first-order theory of the natural numbers with addition (but no multiplication). We characterize sets that can be defined by a Presburger formula as exactly the sets whose characteristic functions can be represented by rational generating functions; a geometric characterization of such sets is also given. In addition, if $\mathbf p=(p_1,\ldots,p_n)$ are a subset of the free variables in a Presburger formula, we can define a counting function $g(\mathbf p)$ to be the number of solutions to the formula, for a given $\mathbf p$. We show that every counting function obtained in this way may be represented as, equivalently, either a piecewise quasi-polynomial or a rational generating function. Finally, we translate known computational complexity results into this setting and discuss open directions.
\end{abstract}

\maketitle

\section{Introduction}
\label{sec1}
A broad and interesting class of sets are those that can be defined over  $\N=\{0,1,2,\ldots\}$ with first order logic and addition.

\begin{definition}A \emph{Presburger formula} is a first-order formula in the language of addition, evaluated over the natural numbers. We will denote a generic Presburger formula as $F(\mathbf u)$, where  $\mathbf u$ are the free variables (those not associated with a quantifier); we use bold notation like $\mathbf u$ to indicate vectors of variables. 

We say that a set $S\subseteq\N^{d}$ is a \emph{Presburger set} if there exists a Presburger formula $F(\mathbf u)$ such that $S=\{\mathbf u\in\N^{d}:\ F(\mathbf u)\}$.

\end{definition}

\begin{example}
\label{ex:PresFormula}
The Presburger formula
\[F(u)=\big(u>1\text{ and }\exists b\in\N:\ b+b+1=u\big)\]
defines the Presburger set $\{3,5,7,\ldots\}$. Since multiplication by an integer is the same as repeated addition, we can conceive of a Presburger formula as a Boolean combination of integral linear (in)equalities, appropriately quantified: $\exists b\ \big(u>1$ and $2b+1=u\big)$.
\end{example}

Presburger proved \cite{Presburger91} that the truth of a Presburger \emph{sentence} (a formula with no free variables) is decidable. We would like to understand more clearly the \emph{structure} of a given Presburger set. One way to attempt to do this is to encode the elements of the set into a {generating function}.

\begin{definition}
\label{def:gf}
Given a set $S\subseteq\N^{d}$, its associated \emph{generating function} is
\[f(S;\mathbf x)=\sum_{\mathbf s\in S}\mathbf x^{\mathbf s}=\sum_{(s_1,\ldots,s_d)\in S}x_1^{s_1}x_2^{s_2}\cdots x_d^{s_d}.\]
\end{definition}

For example, if $S$ is the set defined by Example \ref{ex:PresFormula}, then
\[f(S;x)=x^3+x^5+x^{7}+\cdots=\frac{x^3}{1-x^2}.\]
We see that, in this instance, the generating function has a nice form; this is not a coincidence.

\begin{definition}
\label{def:rgf}
A \emph{rational generating function} is a function that can be written in the form
\[\frac{q(\mathbf x)}{(1-\mathbf x^{\mathbf b_{1}})\cdots(1-\mathbf x^{\mathbf b_{k}})},\]
where $q(\mathbf x)$ is a polynomial in $\Q[\mathbf x]$ and $\mathbf b_{i}\in\N^{d}\setminus\{\mathbf 0\}$.
\end{definition}

We will prove that $S\subseteq\N^{d}$ is a Presburger set if and only if $f(S;\mathbf x)$ is a rational generating function. These are Properties 1 and 3 in the following theorem:
\begin{theorem}
\label{thm:PresSet}
Given a set $S\subseteq\N^{d}$, the following are equivalent:
\begin{enumerate}
\item $S$ is a Presburger set,
\item $S$ is a finite union of sets of the form $P\cap (\lambda+\Lambda)$, where $P$ is a polyhedron, $\lambda\in\Z^d$, and $\Lambda\subseteq\Z^d$ is a lattice.
\item $f(S;\mathbf x)$ is a rational generating function.
\end{enumerate}
\end{theorem}

Property 2 gives a nice geometric characterization of Presburger sets;  the set in Example \ref{ex:PresFormula} can be written as $[3,\infty)\cap (1+2\Z)$.

We are particularly interested in generating functions because of their powerful flexibility: we can use algebraic manipulations to answer questions about the set. For example, $f(S;1,1,\ldots,1)$ is exactly the cardinality of $S$ (if finite). More generally, we may  want to count solutions to a Presburger formula as a function of several parameter variables:

\begin{definition}
The \emph{Presburger counting function} for a  Presburger formula $F(\mathbf c, \mathbf p)$ is
\[g_F(\mathbf p)=\#\{\mathbf c\in\N^{d}:\ F(\mathbf c, \mathbf p)\}.\]
\end{definition}
Note that $\mathbf c$ (the \emph{counted} variables) and $\mathbf p$ (the \emph{parameter} variables) are free variables. We will restrict ourselves to counting functions such that $g_F(\mathbf p)$ is finite for all $\mathbf p\in\N^n$. One could instead either include $\infty$ in the codomain of $g_F$ or restrict the domain of $g_F$ to where $g_F(\mathbf p)$ is finite (the domain would itself be a Presburger set).

A classic example is to take $F(\mathbf c,p)$ to be the conjunction of linear inequalities of the form $a_1c_1+\cdots +a_d c_d \le a_0p$, where $a_i\in \Z$. Then $g_F(p)$ counts the number of integer points in the $p^{\text{th}}$ dilate of a polyhedron.
\begin{example} 
\label{ex:triangle}
If $F(c_1,c_2,p)$ is $2c_1+2c_2\le p$, then the set of solutions $(c_1,c_2)\in\N^2$ lies in the triangle with vertices $(0,0)$, $(0,p/2)$, $(p/2,0)$, and
\begin{align*}g_F(p)&=\frac{1}{2}\left(\floor{\frac{p}{2}}+1\right)\left(\floor{\frac{p}{2}}+2\right)\\
&=\begin{cases}\frac{1}{8}p^2+\frac{3}{4}p+1&\text{if $p$ is even,}\\
\frac{1}{8}p^2+\frac{1}{2}p+\frac{3}{8}&\text{if $p$ is odd.}\end{cases}
\end{align*}
\end{example}

The nice form of this function is also not a coincidence. For this particular type of Presburger formula (dilates of a polyhedron), Ehrhart proved \cite{Ehrhart62} that the counting functions are \emph{quasi-polynomials}:

\begin{definition}
A \emph{quasi-polynomial} (over $\Q$) is a function $g:\N^{n}\rightarrow \Q$ such that there exists an $n$-dimensional lattice $\Lambda\subseteq\Z^{n}$ together with polynomials
$q_{\bar{\lambda}}(\mathbf p)\in\Q[\mathbf p]$, one for each
$\bar{\lambda}\in\Z^n/\Lambda$, such that
\[g(\mathbf p)=q_{\bar{\lambda}}(\mathbf p), \text{ for }\mathbf p\in \bar{\lambda}.\]
\end{definition}
In Example \ref{ex:triangle}, we can take the lattice $\Lambda=2\Z$ and each coset (the evens and the odds) has its associated polynomial.
We need something slightly more general to account for all Presburger counting functions:

\begin{definition}
A \emph{piecewise quasi-polynomial} is a function $g:\N^n\rightarrow \Q$ such that there exists a finite
partition $\bigcup_i (P_i\cap\N^n)$ of $\N^n$ with $P_i$ polyhedra (which may not all be full-dimensional)
and there exist quasi-polynomials $g_i$ such that
\[g(\mathbf p)=g_i(\mathbf p) \text{ for }\mathbf p\in P_i\cap\N^n.\]  
\end{definition}

One last thing that is not a coincidence: For the triangle in Example \ref{ex:triangle}, we can compute
\begin{align*}
\sum_{p\in\N}g_F(p)x^p &= 1+x+3x^2+3x^3+6x^4+\cdots\\
&= \frac{1}{(1-x)(1-x^2)^2},
\end{align*}
a rational generating function! The following theorem says that these ideas are -- almost -- equivalent.

\begin{theorem}
\label{thm:PresFcn}
Given a function $g:\N^{n}\rightarrow \Q$ and the following three possible properties:
\begin{enumerate}
\item[A.] $g$ is a Presburger counting function,
\item[B.] $g$ is a piecewise quasi-polynomial, and
\item[C.] $\sum_{\mathbf p\in\N^{n}}g(\mathbf p)\mathbf x^{\mathbf p}$ is a rational generating function,
\end{enumerate}
we have the implications
\[A\Rightarrow B \Leftrightarrow C.\]
\end{theorem}

\begin{remark}
\label{rmk:IndicatorFcn}
Proving Theorem \ref{thm:PresFcn} will give us much of Theorem \ref{thm:PresSet}, using the following idea.  A set $S\subseteq \Z^d$ corresponds exactly to its characteristic function
\[\chi_S(\mathbf u)=\begin{cases} 1&\text{if }\mathbf u\in S,\\0&\text{if }\mathbf u\notin S.\end{cases}\]
If $S$ is a Presburger set defined by $F(\mathbf u)$, then
\[\chi_S(\mathbf u)=\#\{c\in\N:\ F(\mathbf u)\text{ and }c=0\}\]
is a Presburger counting function.
\end{remark}

In light of Theorem \ref{thm:PresSet}, we might wonder if there is a sense in which $B\Rightarrow A$.  Of course we would have to restrict $g$, for example requiring that its range be in $\N$ (Theorem \ref{thm:PresSet} essentially restricts the range of $g$ to $\{0,1\}$, as it must be a characteristic function). The implication still does not hold, however. For example, suppose the polynomial
\[g(s,t)=(t-s^{2})^{2}\]
were a Presburger counting function given by a Presburger formula $F(\mathbf c, s,t)$, that is,
\[g(s,t)=\#\{\mathbf c\in\N^{d}:\ F(\mathbf c, s,t)\}.\]
 Then the set
\begin{align*}
\big\{(s,t)\in\N^{2}:\ \nexists \mathbf c\ F(\mathbf c, s, t)\big\} &=\{(s,t)\in\N^{2}:\ g(s,t)=0\}\\
&=\{(s,s^{2}):\ s\in\N\}
\end{align*}
would be a Presburger set. This is not the case, however, as it does not satisfy Property 2 in Theorem \ref{thm:PresSet}. If the parameter is univariate, however, the following proposition shows that we do have the implication $B\Rightarrow A$.

\begin{prop}
\label{thm:one_reverse}
Given a function $g:\N\rightarrow \Q$, if $g$ is a piecewise quasi-polynomial whose range is in $\N$, then $g$ is a Presburger counting function.
\end{prop}

In Section \ref{sec4}, we prove Theorem \ref{thm:PresSet}, Theorem \ref{thm:PresFcn}, and Proposition \ref{thm:one_reverse}. In Section \ref{sec2}, we survey related work. In Section \ref{sec3}, we present the primary tools we need for the proofs. In Section \ref{sec5}, we turn to computational questions; this section surveys known results, but restates them in terms of Presburger arithmetic.

\section{Related work}
\label{sec2}
Presburger arithmetic is a classical first order theory of logic, proven decidable by Presburger \cite{Presburger91}. Bounds on the complexity of decision algorithms have been examined in the general case \cite{Berman80,Cooper72,FR79,FR73,Oppen78}, as a function of the number of quantifier alternations \cite{Furer82,Gradel88,Haase14,RL78,Schoning97}, and for fixed (small) number of quantifier alternations \cite{Gradel89,Scarpellini84}.

A finite automata approach to Presburger arithmetic was pioneered in \cite{Buchi60,Cobham69}, and continues to be an active area of research (see, for example, \cite{BC96,CJ98,Klaedtke08,WB95}). This approach is quite different from the present paper's, but it can attack similar questions: for example, see \cite{PC04} for results on counting solutions to Presburger formulas (non-parametrically).

The importance of understanding Presburger Arithmetic is highlighted by the fact that many problems in computer science and mathematics can be phrased in this language: for example, integer programming \cite{Lenstra83,Schrijver03}, geometry of numbers \cite{Cassels97,Kannan90}, Gr\"obner bases and algebraic integer programming \cite{Sturmfels96,Thomas03}, neighborhood complexes and test sets \cite{Scarf97,Thomas95}, the Frobenius problem \cite{RA05}, Ehrhart theory \cite{BR07,Ehrhart62}, monomial ideals \cite{MS05}, and toric varieties \cite{Fulton93}. Several of the above references analyze the computational complexity of their specific problem. In most of the above references, the connection to Presburger arithmetic is only implicit. 

The algorithmic complexity of specific rational generating function problems has been addressed in, for example, \cite{Barvinok94,BW03,BGP11,DHH03,GM11,HS03}. Several of these results are summarized in Section \ref{sec5} of this current paper.

Connections between subclasses of Presburger arithmetic and generating functions are made explicit in \cite{Barvinok06,BP99,BW03}. Connections between rational generating functions and quasi-polynomials have been made in \cite{Ehrhart62,Stanley80,Sturmfels95}, and the algorithmic complexity of their relationship was examined in \cite{VW04}. 
Counting solutions to Presburger formulas has been examined in \cite{P94}, though the exact scope of the results is not made explicit, and rational generating functions are not used. Similar counting algorithms appear in \cite{CL98}, and \cite{DIV09} proves that the counting functions for a special class of Presburger formuals (those whose parameters $\mathbf p$ only appear in terms $c_i\le p_i$) are piecewise quasi-polynomials. This current paper is the first to state and prove a general connection between Presburger arithmetic, quasi-polynomials, and rational generating functions.

Theorem \ref{thm:PresSet} was originally proved in the author's thesis \cite{WoodsThesis04}; in this paper, it is put into context as a consequence of the more general Theorem \ref{thm:PresFcn}. A simpler geometric characterization  of Presburger sets (equivalent to Property 2 of Theorem \ref{thm:PresSet})   was given in \cite{GS66}: they are the \emph{semi-linear} sets, those sets that can be written as a finite union of sets of the form $S=\{\mathbf a_{0}+\sum_{i=1}^{k}n_i\mathbf a_{i}:\ n_i\in\N\}$, where $\mathbf a_{i}\in\N^d$. Furthermore, if one takes these $S$ to be disjoint and requires the $\mathbf a_{1},\ldots,\mathbf a_{k}$ to be linearly independent, for each $S$ (as \cite{GS66} implicitly prove can be done, made explicit in \cite{DIV09} as \emph{semi-simple} sets), then each $S$ can be encoded with the rational generating function \[\frac{\mathbf x^{\mathbf a_{0}}}{\left(1-\mathbf x_{1}^{\mathbf a_{1}}\right)\cdots\left(1-\mathbf x_{k}^{\mathbf a_{k}}\right)}\]
and we obtain a slightly different version of $2\Rightarrow 3$ in Theorem \ref{thm:PresSet}. There seems to be no previous result analogous to $3\Rightarrow 2$.

An extended abstract of this current paper appeared in ICALP 2013.

\section{Primary background theorems}
\label{sec3}

Here we detail several tools we will use. When these tools have algorithmic (polynomial-time) versions, we mention them too, which will help our discussion in Section \ref{sec5}.

The first tool we need is a way to simplify Presburger formulas. As originally proved \cite{Presburger91} by Presburger (see \cite{Oppen78} for a nice exposition), we can completely eliminate the quantifiers if we are allowed to also use modular arithmetic.
\begin{definition}
An \emph{extended Presburger formula} is a Boolean formula with variables in $\N$ expressible in the elementary language of Presburger Arithmetic extended by the $\bmod\, k$ operations, for constants $k>1$. \end{definition}

\begin{theorem}
\label{thm:PresElim}
Given a formula $F(\mathbf u)$ in extended Presburger arithmetic (and hence given any formula in Presburger arithmetic), there exists an equivalent \emph{quantifier-free} formula $G(\mathbf u)$ such that
\[\{\mathbf u\in\N^{d}:\ F(\mathbf u)\}=\{\mathbf u\in\N^{d}:\ G(\mathbf u)\}.\]
\end{theorem}
For instance, the set from Example \ref{ex:PresFormula} can be written as $(u> 1 \text{ and }u\bmod 2 =1)$.

Next, we give two theorems that tie in generating functions. The first gives us a way to convert from a specific type of Presburger set to a generating function.

\begin{theorem}
\label{thm:gfpolyhedra} Given a point $\lambda\in\Z^d$, a lattice $\Lambda\subseteq\Z^{d}$, and a rational polyhedron $P\subseteq\R_{\ge 0}^{d}$, $f\big(P\cap(\lambda+\Lambda);\mathbf x\big)$ (as given in Definition \ref{def:gf})  is a rational generating function. In addition, for fixed $d$, this rational generating function can be found in polynomial time.
\end{theorem}

The first step to proving this is to use Brion's Theorem \cite{Brion88}, which says that the generating function can be decomposed into functions of the form $f\big(K\cap(\lambda+\Lambda);\mathbf x\big)$, where $K$ is a cone. Then, one can notice that integer points in cones have a natural structure that can be encoded as geometric series.

\begin{example}
Let $K\subseteq\R^2$ be the cone with vertex at the origin and extreme rays $\mathbf u=(1,0)$ and $\mathbf v=(1,2)$. Using the fact that the lattice $(u\Z+v\Z)$ has index 2 in $\Z^2$, with coset representatives $(0,0)$ and $(1,1)$, every integer point in $K$ can be written as either $(0,0)+\lambda_1\mathbf u+\lambda_2\mathbf v$ or $(1,1)+\lambda_1\mathbf u+\lambda_2\mathbf v$, where $\lambda_1,\lambda_2\in\N$. Therefore
\begin{align*}
f(K\cap\Z^2;\mathbf x)&=(\mathbf x^{(0,0)}+\mathbf x^{(1,1)})(1+\mathbf x^\mathbf u+\mathbf x^{2\mathbf u}+\cdots)(1+\mathbf x^{\mathbf v}+\mathbf x^{2\mathbf v}+\cdots)\\
&=\frac{\mathbf x^{(0,0)}+\mathbf x^{(1,1)}}{(1-\mathbf x^\mathbf u)(1-\mathbf x^{\mathbf v})}.
\end{align*}
\end{example}

See \cite[Chapter~VIII]{Barvinok02}, for example, for more details.

The complexity version of Theorem \ref{thm:gfpolyhedra} is due to Barvinok \cite{Barvinok94}. For polynomial-time algorithms yielding generating functions, the output will generally be represented as a \emph{sum} of the basic rational functions given in Definition \ref{def:rgf}; getting a common denominator and simplifying into one fraction may take exponential time. Since the dimension is fixed, the input polyhedra may equivalently be given with either a vertex or hyperplane description.

Next, we would like to be able to perform substitutions on the variables in a rational generating function and still retain a rational generating function; particularly, we would like to substitute in 1's for several of the variables.

\begin{theorem}
\label{thm:subst}
Given a rational generating function $f(\mathbf x)$, then
\[g(\mathbf z)=f(\mathbf z^{\mathbf l_{1}},\mathbf z^{\mathbf l_{2}},\ldots,\mathbf z^{\mathbf l_{d}}),\]
with $\mathbf l_i\in\N^k$,
is also a rational generating function, assuming the substituted values do not lie entirely in the poles of $f$. In particular, substituting in $x_i=\mathbf z^{\mathbf 0}=1$ yields a rational function, if $x_i=1$ is not entirely in the poles of $f$.

Furthermore, for a fixed bound on the number of binomials in the denominator of $f$, there is a polynomial time algorithm to compute $g$ as a rational generating function.
\end{theorem}

The existence version of this theorem is immediate: if substituting in $x_i=\mathbf z^{\mathbf l_i}$ would make any of the binomials in the denominator of $f$ zero (when $f$ is written in the form from Definition \ref{def:rgf}), then that binomial must be a factor of the numerator (or else $x_i=\mathbf z^{\mathbf l_i}$ would lie entirely in the poles of $f$); therefore, substituting in  $x_i=\mathbf z^{l_i}$ yields a new rational generating function. The complexity version is more difficult and is due to Barvinok and Woods \cite{BW03}. The problem is that we may have the rational generating function input as a sum of rational functions. Unfortunately, the substituted values may be in the poles of some of the terms of the sum without being in the poles of the entire generating function; careful limits must be taken. 

\begin{example}
Suppose we would like to substitute $x=1$ into
\[f(x)=\frac{1}{1-x}-\frac{x^{1000}}{1-x}.\]
Combining and simplifying to $f(x)=1+x+\cdots+x^{999}$ is horribly inefficient. Since $x=1$ is a pole of both fractions, we must be careful: Consider the respective Laurent series expansions of the terms:
\begin{align*}\frac{1}{1-x}&=c_{-1}(x-1)^{-1}+c_{0}(x-1)^0+c_1(x-1)^1+\cdots,\\
 -\frac{x^{1000}}{1-x}&=d_{-1}(x-1)^{-1}+d_{0}(x-1)^0+d_1(x-1)^1+\cdots.
 \end{align*}
 We know that $1$ is not a pole of $f(x)$, so $c_{-1}+d_{-1}=0$. Substituting in $x=1$ into the sum, we will get simply $c_0+d_0$. Carefully calculating $c_0$ and $d_0$ will give us $f(1)=1000$.
\end{example}

Finally, we need a connection between Presburger formulas and quasi-poly\-no\-mials. This is given by Sturmfels \cite{Sturmfels95}:

\begin{definition}
Given $\mathbf a_{1},\ldots,\mathbf a_{d}\in\N^{n}$, the \emph{vector partition function} $g:\N^{n}\rightarrow\N$ is defined by
\[g(\mathbf p)=\#\{(\lambda_{1},\ldots,\lambda_{d})\in\N^{d}:\ \mathbf p=\lambda_{1}\mathbf a_{1}+\cdots+\lambda_{d}\mathbf a_{d}\},\]
that is, the number of ways to partition the vector $\mathbf p$ into parts taken from $\{\mathbf a_i\}$.
\end{definition}

\begin{theorem}
\label{thm:vpf}
Any vector partition function is a piecewise quasi-polynomial.
\end{theorem}

See \cite{Beck04} for a self-contained explanation utilizing the partial fraction expansion of the generating function
\[\sum_{\mathbf p\in\N^n} g(\mathbf p)\mathbf x^{\mathbf p} = \frac{1}{(1-\mathbf x^{\mathbf a_1})\cdots(1-\mathbf x^{\mathbf a_d})};\]
this equality can be obtained by rewriting the rational function as a product of infinite geometric series:
\[(1+\mathbf x^{\mathbf a_1}+\mathbf x^{2\mathbf a_1}+\cdots)\cdots(1+\mathbf x^{\mathbf a_d}+\mathbf x^{2\mathbf a_d}+\cdots).\]
For example, if $a_1=1$ and $a_2=a_3=2$, then the vector partition function is encoded by the generating function
\[\frac{1}{(1-x)(1-x^2)^2}.\]
We saw previously that this generating function corresponds to the quasi-poly\-no\-mial in Example \ref{ex:triangle}.

 A complexity version of Theorem \ref{thm:vpf} is discussed in Section \ref{subsect:BC}; an appropriate representation for piecewise quasi-poly\-nomials must first be decided on.

\section{Proofs}
\label{sec4}

\subsection{Proof of Theorem \ref{thm:PresFcn}}

\noindent$\mathbf{A\Rightarrow C.}$

Given a Presburger counting function, $g(\mathbf p)=\#\{\mathbf c\in\N^{d}:\ F(\mathbf c, \mathbf p)\}$, we first apply Presburger Elimination (Theorem \ref{thm:PresElim}) to $F$ to obtain a quantifier free formula, $G(\mathbf c, \mathbf p)$, in extended Presburger arithmetic such that $g(\mathbf p)=\#\{\mathbf c\in\N^{d}:\ G(\mathbf c, \mathbf p)\}$. Integers which satisfy a statement of the form \[a_1p_1+\cdots+a_np_n+a_{n+1}c_1+\cdots +a_{n+d}c_d\equiv a_0 \mod{m}\] are exactly sets $\lambda+\Lambda$, where $\lambda\in\Z^{n+d}$ and $\Lambda$ is a lattice in $\Z^{n+d}$. Since $G(\mathbf c, \mathbf p)$ is a Boolean combination of linear inequalities and these linear congruences, we may write the set, $S$, of points $(\mathbf c,\mathbf p)$ which satisfy $G(\mathbf c, \mathbf p)$ as a \emph{disjoint} union
\[S=\bigcup_{i=1}^k P_i\cap(\lambda_i+\Lambda_i),\]
where, for $1\le i\le k$, $P_i\subseteq\R_{\ge 0}^{n+d}$ is a polyhedron,
$\Lambda_i$ is a sublattice of $\Z^{n+d}$, and $\lambda_i$ is in $\Z^{n+d}$. (To see this, convert the formula into disjunctive normal form; each conjunction will be of this form $P_i\cap(\lambda_i+\Lambda_i)$; these sets may overlap, but their overlap will also be of this form.)

Let $S_i=P_i\cap(\lambda_i+\Lambda_i)$. By Theorem \ref{thm:gfpolyhedra}, we know we can write $f(S_i;\mathbf y, \mathbf x)$ as a rational
generating function,
and so
\[f(S;\mathbf y,\mathbf x)=\sum_i f(S_i;\mathbf y,\mathbf x)= \sum_{(\mathbf c,\mathbf p):\ G(\mathbf c, \mathbf p)} \mathbf y^{\mathbf c}\mathbf x^{\mathbf p}\]
can be written as a rational generating function. Finally, we substitute $\mathbf y=(1,1,\ldots,1)$, using Theorem \ref{thm:subst}, to obtain the rational generating function
\[\sum_{\mathbf p} \#\{\mathbf c\in\N^{d}:\ G(\mathbf c, \mathbf p)\}\mathbf x^{\mathbf p}=\sum_{\mathbf p} g(\mathbf p)\mathbf x^{\mathbf p}.\]

\noindent$\mathbf{C\Rightarrow B.}$

It suffices to prove this for functions $g$ such that
$\sum_{\mathbf p}g(\mathbf p)\mathbf x^{\mathbf p}$ is a rational generating function of the form
\[\frac{\mathbf x^{\mathbf q}}{(1-\mathbf x^{\mathbf a_1})(1-\mathbf x^{\mathbf a_2})\cdots(1-\mathbf x^{\mathbf a_k})},\]
where $\mathbf q\in\N^n,\mathbf a_i\in\N^n\setminus\{0\}$, because the property of being a piecewise
quasi-polynomial is preserved under linear combinations.
Furthermore, we may take $\mathbf q=(0,0,\ldots,0)$, because multiplying by $\mathbf x^{\mathbf q}$ only
shifts the domain of the function $g$.  Expanding this rational generating function as a product of infinite geometric series,
\[\sum_{\mathbf p}g(\mathbf p)\mathbf x^{\mathbf p}=(1+\mathbf x^{\mathbf a_1}+\mathbf x^{2\mathbf a_1}+\cdots)\cdots(1+\mathbf x^{\mathbf a_k}+\mathbf x^{2\mathbf a_k}+\cdots),\]
and we see that
\[g(\mathbf p)=\#\{(\lambda_{1},\ldots,\lambda_{k})\in\N^{k}:\ \mathbf p=\lambda_{1}\mathbf a_{1}+\cdots+\lambda_{k}\mathbf a_{k}\}.\]
This is exactly a vector partition function, which Theorem \ref{thm:vpf} tells us is a piecewise quasi-polynomial.

\noindent$\mathbf{B\Rightarrow C.}$ 

Any piecewise quasi-polynomial can be written as a linear combination of functions of the form
\[g(\mathbf p)=\begin{cases} \mathbf p^{\mathbf a} &\text{if }\mathbf p\in P\cap(\lambda+\Lambda),\\
0 & \text{otherwise},\end{cases}\]
where $\mathbf a\in \N^n$, $P\subseteq\R^n_{\ge 0}$ is a polyhedron, $\lambda\in\Z^n$, and $\Lambda$ is a sublattice of $\Z^n$. Since linear combinations of rational generating functions are rational generating functions, it suffices to prove it for such a $g$.
Let $c_{ij}$, for $1\le i\le n$ and $1\le j\le a_i$, be variables, and define the polyhedron
\begin{align*}Q=\{(\mathbf p, \mathbf c)&\in\N^{n+a_1+\cdots+a_n}:\\
&\mathbf p\in P \text{ and }1\le c_{ij}\le p_i\text{ for all }c_{ij}\}.
\end{align*}
This $Q$ is defined so that $\#\{\mathbf c:\ (\mathbf p, \mathbf c)\in Q\}$ is $p_1^{a_1}\cdots p_n^{a_n}=\mathbf p^{\mathbf a}$ for $\mathbf p\in P$ (and 0 otherwise). Using Theorem \ref{thm:gfpolyhedra}, we can find the generating function for the set $Q\cap(\lambda+\Lambda)$ as a rational generating function. Substituting $\mathbf c=(1,1,\ldots,1)$, using Theorem \ref{thm:subst}, gives us
$\sum_{\mathbf p}g(\mathbf p)\mathbf x^{\mathbf p}$ as a rational generating function.

\subsection{Proof of Theorem \ref{thm:PresSet}}

Given a set $S\subseteq \Z^d$, define the characteristic function, $\chi_S: \N^d\rightarrow \{0,1\}$, as in Remark \ref{rmk:IndicatorFcn}. Define a new property:
\begin{enumerate}
\item[$2'$.] $\chi_S$ is a piecewise quasi-polynomial.
\end{enumerate}
Translating Theorem \ref{thm:PresFcn} into properties of $S$ and $\chi_S$, we have
\[1\Rightarrow (2' \Leftrightarrow 3).\]
So we need to prove $2\Rightarrow 1$ and $2'\Rightarrow 2$.

\medskip

\noindent$\mathbf{2\Rightarrow 1.}$

This is straightforward: the property of being an element of $\lambda+\Lambda$ can be written using linear congruences and existential quantifiers, and the property of being an element of $P$ can be written as a set of linear inequalities.

\medskip

\noindent$\mathbf{2'\Rightarrow 2.}$

Since $\chi_S$ is a piecewise quasi-polynomial, it is constituted from associated polynomials. Let us examine such a polynomial $q(\mathbf p)$ that agrees with $\chi_S$ on some $P\cap (\lambda+\Lambda)$, where $P\subseteq{\R_{\ge 0}^n}$ is a polyhedron, $\lambda\in\Z^n$, and $\Lambda$ a sublattice of $\Z^n$.  It suffices to prove that 2 holds for $S\cap P\cap (\lambda+\Lambda)$, since $S$ is the disjoint union of such pieces.

Ideally, we would like to argue that, since $q$ only takes on  the values 0 and 1, the polynomial $q$ must be constant on $P\cap (\lambda+\Lambda)$, at least if $P$ is unbounded. This is not quite true; for example, if \[P=\big\{(x,y)\in\R^2:\ x\ge 0 \text{ and }0\le y \le
1\big\},\] then the polynomial $q(x,y)=y$ is 1 for $y=1$ and 0 for $y=0$.

What we can say is that $q$ must be constant on any infinite ray contained in $P\cap (\lambda+\Lambda)$: if we parametrize the ray by $\mathbf x(t)=(x_1(t),\cdots, x_n(t))$, then $q(\mathbf x(t))$ is a \emph{univariate} polynomial that is either 0 or 1 at an infinite number of points, and so must be constant. Inductively, we can similarly show that $q$ must be constant on any cone contained in $P$.

Let $K$ be the cone with vertex at the origin \[K=\{\mathbf y\in\R^n:\
\mathbf y+P\subseteq P\}.\] Then $K$ is the largest cone such that
the cones $\mathbf x+K$ are contained in $P$, for all $\mathbf x\in P$; $K$ is often called the
\emph{recession cone} or \emph{characteristic cone} of $P$ (see
Section 8.2 of \cite{Schrijver86}), and the polyhedron $P$ can be decomposed into a Minkowski sum $K+Q$, where $Q$ is a \emph{bounded} polyhedron. We can write $P\cap(\lambda+\Lambda)$
as a finite union (possibly with overlap) of sets of the form
\[Q_j=(v_j+K)\cap(\lambda+\Lambda),\]
for some $v_j$, and on each of these pieces $q$
must be constant. If $q$ is the constant 1 on $Q_j$, then $Q_j$ is contained in $S$, and if $q$ is the constant 0, then none of $Q_j$ is in $S$. Since $S$ is a finite union of the appropriate $Q_j$,  $S$ has the form needed for Property 2.

\subsection{Proof of Proposition \ref{thm:one_reverse}}
Given that $g$ is a piecewise quasi-poly\-no\-mial with range in $\N$, we must find a Frobenius formula $F(\mathbf c,p)$ such that $g(p)=\#\{\mathbf c\in\N^d:\ F(\mathbf c,p)\}.$
It suffices to find an $F$ that agrees with $g$ for sufficiently large $p$, because for any finite set $\{p_i\}$ we may include
\[(p=p_i)\Rightarrow \big(1\le c_1\le g(p_i)\wedge c_2=\cdots=c_d=0\big) \]
in the creation of $F$, so that the number of $\mathbf c$ satisfying $F$ is exactly $g(p_i)$.

Since the domain of $g$ is one-dimensional, a piecewise quasi-polynomial can be thought of as a function that is a quasi-polynomial for sufficiently large $p$. Therefore we may assume, without loss of generality, that $g$ is a quasi-polynomial. Each lattice coset may be handled separately and combined to form the final formula, $F$. Therefore we may assume that $g$ is a polynomial.

Let $m$ be the least common multiple of the denominators of the coefficients of $g\in\Q[p]$. Suppose $p=mb+i$, for some $b$ (which can be encoded in a Presburger formula, with $m$ and $i$ constants and $b$ a bound variable). Considering $(mb+i)^k$ as a polynomial in $b$, all coefficients but (possibly) the constant term are multiples of $m$, so considering $g(mb+i)$ as a polynomial in $b$, all coefficients but (possibly) the constant term will be integral. Since $g(mb+i)$ is integer-valued, the constant term must also be an integer, that is, $g(mb+i) \in\Z[b]$. Since each residue class of $p$ mod $m$ can be handled separately and combined at the end, we may assume without loss of generality that $m=1$ and $g\in\Z[p]$.

Let $g(p)=\sum_{i=0}^k a_ip^i$, $a_i\in\Z$. Since $g(p)\ge 0$ for all $p$, the leading coefficient, $a_k$, must be positive. Then $a_kp^k$ is the number of $\mathbf c\in\N^{k+1}$ such that
\[(1\le c_0\le a_k)\wedge(1\le c_1\le p)\wedge\cdots\wedge(1\le c_k\le p).\]
All other \emph{positive} terms of $g$ can be similarly encoded as counting solutions to a Presburger formula (carefully defined so the solution sets are disjoint). To take care of a \emph{negative} term $a_ip^i$, take $p$ sufficiently large so that $a_kp^k$ sufficiently dominates all other terms (as before, the finite set of smaller $p$ can be dealt with separately); now remove $a_ip^i$ points from the set of solutions. For example,
\begin{align*}\neg\Big((c_0&=c_{i+2}=\cdots=c_k=1)\wedge(1\le c_1\le a_i)\\
&\wedge(1\le c_2\le p)\wedge\cdots\wedge(1\le c_{i+1}\le p)\Big)
\end{align*}
removes $a_ip^i$ solutions, as long as $p\ge a_i$ (to do this for multiple $i$, define these subtractions carefully so that the solution sets do not overlap). All together, this give us our desired $F(\mathbf c, p)$ certifying that $g$ is a Presburger counting function.

\section{Computational aspects}
\label{sec5}
We first describe computational aspects of $A\Rightarrow (B\text{ and }C)$. Many partial results are known here, but there are several open problems. Then we discuss $B\Leftrightarrow C$, which works nicely: not only are $B$ and $C$ logically equivalent, but, after we make sense of the problem, they are computationally equivalent (in fixed dimension). Finally, we discuss other related complexity results.

\subsection{$A\Rightarrow (B\text{ and }C)$}
\label{subsec:AC}
Note that another reasonable class of problems to study would be
sentences where we are also allowed multiplication of variables, for
example
\[\exists a\in\N, \exists b\in\N:\  a^2+2b^2\le 31.\] These problems, however, are very hard: undecidability, in the general case, is a consequence of G\"odel's First Incompleteness Theorem \cite{Godel31}. Even if we focus on sentences with only existential quantifiers and a fixed number of variables (restrictions that we shall see work well for Presburger arithmetic), it is still undecidable: In
fact, there is a certain multivariate polynomial
$p(x_0,x_1,\ldots,x_d)$ such that the class of
problems \begin{align*}\text{Given }&a\in\N,\text{ decide whether } \exists b_1,\exists b_2,\ldots,\exists b_d:\\
&p(a,b_1,b_2,\ldots,b_d)=0, \text{ with }b_i\in\N\end{align*} is
undecidable.  This is a consequence of the DPRM-theorem (after
Davis, Putnam, Robinson, and Matiyasevich, see, for example,
\cite{Davis73}), which solves Hilbert's 10th problem in the
negative. 

The Presburger arithmetic problem (deciding whether a Presburger sentence is true or false) is, at least, decidable, as
originally proved \cite{Presburger91} by Presburger in 1929. Since
then, better algorithms have been found.  For example, D. Oppen gave
an algorithm \cite{Oppen78}, based on work of D. Cooper
\cite{Cooper72}, with running time $2^{2^{2^{c\phi}}}$, where $\phi$
is the input size of the problem and $c$ is a constant.
Nevertheless, Fischer and Rabin gave lower bounds on the running time \cite{FR73}: any algorithm that solves all Presburger arithmetic
problems will sometimes take at least $2^{2^{c'\phi}}$ steps, where
$c'$ is a constant.

Among the simplest Presburger sentences are integer programming problems, which are themselves NP-complete if the dimension is not fixed. This suggests that any interesting polynomial time results will require a fixed number of variables.

We discuss three types of problems for a Presburger formula $F(\mathbf u)$:
\begin{itemize}\item The \emph{decision} problem: Is  $\exists \mathbf u\ F(\mathbf u)$ true or false?
\item The \emph{counting} problem: How many solutions $\mathbf u$ to $F(\mathbf u)$ are there?
\item The \emph{generating function} problem: Compute $\sum_{\mathbf u:\ F(\mathbf u)}\mathbf x^{\mathbf u}$ as a rational generating function.
\end{itemize}

Clearly, the counting problem is harder than the decision problem. Now notice that if $\sum_{\mathbf u:\ F(\mathbf u)}\mathbf x^{\mathbf u}$ can be computed in polynomial time, then the counting problem $\#\mathbf u\ F(\mathbf u)$ can be solved in polynomial time: use Theorem \ref{thm:subst} to plug in $\mathbf x=(1,1,\ldots, 1)$. Consequently, the generating function problem is the hardest of the three and the decision problem $\exists \mathbf u\ F(\mathbf u)$ is the easiest. Historically, polynomial time results have been discovered first for the decision problems, and these solutions have not needed generating functions. However, polynomial time results for counting problems have generally required the full power of generating functions. We now discuss several of these results.

First we examine, for a fixed number of variables and \emph{quantifier-free} $F(\mathbf u)$, deciding $\exists \mathbf u\ F(\mathbf u)$. This is effectively Lenstra's algorithm for integer programming  \cite{Lenstra83}, which decides whether there exists a $\mathbf u$ satisfying the conjunction of several linear inequalities. The only additional step we need is to note that a Boolean combination of linear inequalities can be put in disjunctive normal form in polynomial time. This is slightly surprising, because it is not doable for general Boolean functions: for example, the disjunctive normal form for $\bigwedge_{i=1}^n(A_{i1}\vee A_{i2})$ requires $2^n$ disjunctions.

\begin{prop}
\label{PropDNF} Fix $d$.  There is a polynomial time algorithm
which, given a quantifier-free formula $F(u_1,u_2,\ldots,u_d)$
consisting of linear inequalities and Boolean operations, converts
$F$ into disjunctive normal form.
\end{prop}

\begin{proof}
The inequalities appearing in $F$ cut $\R^d$ into many polyhedral
pieces. It might appear at first glance that the number of such
pieces could be $2^N$, where $N$ is the total number of inequalities
in $F$.  Nevertheless, the
number of pieces is bounded by
\[\Phi(d,N)=\binom{N}{0}+\binom{N}{1}+\cdots+\binom{N}{d}.\]
See Section 6.1 of \cite{Matousek02} for a proof by induction.  We have that
$\Phi(d,N)$ is a polynomial in $N$ of degree at most $d$ (where $d$
is fixed), and following the inductive proof, we see that the description of each piece (and the lower-dimensional polyhedral pieces on their boundaries) may be found in
polynomial time. Within the interior of  each piece, $F$ is either always true or
always false. Then the disjunctive normal form is simply $\bigvee_P
\{x\in \interior(P)\},$ where the disjunction is taken over all polyhedral
pieces $P$ (including lower-dimensional ones) such that $F$ is true.
\end{proof}

The counting and generating function problems for quantifier-free Presburger formulas were proven to be polynomial time by Barvinok \cite{Barvinok94}.

Adding one quantifier alternation makes this problem hard, unfortunately: even with two variables, deciding $\exists a\forall b\ F(a,b)$ is NP-hard, as Sch\"oning showed \cite{Schoning97}. In general, Gr\"adel showed \cite{Gradel88} that these problems rapidly increase in difficulty with increased quantifier alternation, following the polynomial-time hierarchy.

The above discussion demonstrates that, if one considers only fixing the number of quantifier alternations or number of variables, the tractability of Presburger arithmetic is understood (and see \cite{Haase14}, where Haase gives a few more variations). A possible modification is to also fix the number of linear inequalities allowed. In this case, Kannan proved \cite{Kannan90} that the decision problem $\exists \mathbf a\forall \mathbf b\ F(\mathbf a, \mathbf b)$ can be solved in polynomial time, and Barvinok and Woods proved \cite{BW03} that the counting and generating function problems could be solved in polynomial time.
 
With more quantifier alternation, the computational complexity of the problem is unknown.

\begin{conj}
For a fixed number of variables and linear inequalities, there exists a polynomial time algorithm to compute the generating function for a Presburger set. Hence the counting problem and the decision problem may be answered in polynomial time.
\end{conj}
 
\subsection{$B\Leftrightarrow C$}
\label{subsect:BC}
The first question here is how to interpret the problem: how should a quasi-polynomial be represented? The obvious way, listing every constituent polynomial, is inefficient. For example, the generating function $1/(1-x^n)$ corresponds to the quasi-polynomial $g(x)=1$ if $n\big|x$ and $g(x)=0$ otherwise; this requires $n$ polynomials to represent. The solution of Verdoolaege and Woods \cite{VW04} is to use \emph{piecewise step-polynomials}:

\begin{definition}
A {\em step-polynomial\/} $g : \N^n \to \Q$ is a function written in
the form
\[
g(\mathbf p) = \sum_{i=1}^{m}\alpha_i\prod_{j=1}^{d_i} \floor{a_{ij1}p_1+\cdots+a_{ijn}p_n+b_{ij}},
\]
where $\alpha_i,a_{ijk},b_{ij}\in\Q$ and
$\floor{\cdot}$ is the greatest integer function.
\end{definition}

For example, the counting function corresponding to $1/(1-x^n)$ can be written at the step-polynomial $\floor{\frac{x}{n}}-\floor{\frac{x-1}{n}}$.

\begin{definition}A {\em piecewise step-polynomial} is a function $g:\N^n\rightarrow\Q$ such that there exists a finite partition $\bigcup_i (P_i\cap\N^n)$ of $\N^n$ with $P_i$ polyhedra (which may not all be full dimensional) and there exists step-polynomials $g_i$ such that
\[g(\mathbf p)=g_i(\mathbf p) \text{ for }\mathbf p\in P_i\cap\N^n.\]
\end{definition}

 Verdoolaege and Woods prove \cite{VW04} that one may convert between rational generating functions and piecewise step-polynomials in polynomial time (for fixed degree of the step-polynomial). Therefore, the conceptual equivalence, $B\Leftrightarrow C$, is also an algorithmic equivalence, provided piecewise quasi-polynomials are suitably represented.

\subsection{Other Complexity Relationships}
One reason generating functions are so valuable is that algebraic manipulations can be brought into play. We've already seen that substituting in 1 for variables can be done in polynomial time, enabling us to compute cardinality. Several other interesting operations have polynomial time algorithms.

For example, given the generating functions for two sets $S$ and $T$, the generating functions for $S\cap T$, $S\cup T$, and $S\setminus T$ can all be computed in polynomial time \cite{BW03} (assuming fixed number of variables and fixed bound on the number of binomials in the denominators of the generating functions). This means that the Boolean operations and/or/not can be mirrored by generating functions. The key ingredient is the Hadamard product, which can be computed in polynomial time \cite{BW03}.

\begin{definition}
The \emph{Hadamard product} of
\[f(\mathbf x)=\sum_{\mathbf p\in\N^n} c(\mathbf p)\mathbf x^\mathbf p\text{ and }g(\mathbf x)=\sum_{\mathbf p\in\N^n} d(\mathbf p)\mathbf x^\mathbf p\]
is
\[(f\star g)(\mathbf x)=\sum_{\mathbf p\in\N^n}c(\mathbf p)d(\mathbf p)\mathbf x^\mathbf p.\]
\end{definition}

For example, the fact that $f(S\cap T; \mathbf x)=f(S;\mathbf x)\star f(T;\mathbf x)$ immediately shows us how to compute the generating function for $S\cap T$, given $f(S;\mathbf x)$ and $f(T;\mathbf x)$.

Unfortunately, the operation of quantification cannot be similarly mirrored efficiently with generating functions: if we are given $f(S;\mathbf x,y)$ for some $S\subseteq\N^{d+1}$, and define $T=\{\mathbf a\in\N^d:\ \forall b\in \N,\  (\mathbf a, b)\in S\}$, then obtaining $f(T;\mathbf x)$ is NP-hard. To see this, take $S$ to be the Presburger set for a quantifier free formula, $F(a,b)$. We can compute $f(S;x,y)$ in polynomial time, as discussed in Section \ref{subsec:AC}. If we could compute $f(T;x)$ in polynomial time, then we would be able to decide $\exists a\forall b\ F(a,b)$ in polynomial time, by checking if $f(T;1)\ne 0$. But deciding this is NP-hard \cite{Schoning97}, as mentioned in Section \ref{subsec:AC}.

Even simpler sounding tasks can be surprisingly difficult. There is no presently-known ``nice'' way to tell whether a rational generating function, $g(\mathbf x)$ (input as a sum of the basic rational functions from Definition~\ref{def:rgf}), is identically zero. A heavy-handed -- but polynomial time -- algorithm is, after checking that $\mathbf 1$ is not a pole of $g$, to substitute $\mathbf 1$ into $g\star g$: the coefficients of $g\star g$ are nonnegative, so $(g\star g)(\mathbf 1)=0$ if and only if $g=0$. 

\section*{Acknowledgements}
My thanks to the referees for pointers to relevant related works. I am also grateful to them for helping me talk like a logician and like a theoretical computer scientist (when, in fact, I am a combinatorist).

\end{document}